\documentclass{amsart}
\usepackage{latexsym,amssymb,amsmath,amsthm,amscd}
\newtheorem{lemma}{Lemma}
\newtheorem{theorem}{Theorem}

\newtheorem{conjecture}{Conjecture}

\newtheorem{remark}{Remark}
\begin{document}
\title{On the $\ell_p$ and $\ell_{p,q}$ norms of Cauchy--Toeplitz matrices}
\author[Ts. Batbold]{Tserendorj Batbold}
\address{
Department of Mathematics\\
National University of Mongolia\\
Ulaanbaatar 14201\\
Mongolia}
\email{tsbatbold@hotmail.com}
\dedicatory{This paper is dedicated to the memory of my father and teacher, Vandansanj Tserendorj.}

\begin{abstract} New upper and lower bounds for the $\ell_p (1<p<\infty)$ norms of  Cauchy--Toeplitz matrices in the form $T_n=[2/(1+2(i-j))]_{i,j=1}^n$ are derived.
    Moreover, we give a complete answer to a conjecture proposed by D. Bozkurt.
\end{abstract}
\maketitle
\section{Introduction}
Let $C=\left[1 /\left(x_i-y_j\right)\right]_{i, j=1}^n\left(x_i \neq y_j\right)$ be a Cauchy matrix, and $T_n=\left[t_{j-i}\right]_{i, j=0}^n$ be a Toeplitz matrix. In generally Cauchy--Toeplitz matrix is being defined as
$$
T_n=\left[\frac{1}{g+(i-j) h}\right]_{i, j=1}^n,
$$
where $h \neq 0, g$ and $h$ are any numbers and $\frac{g}{h}$ is not integer. If we substitute $g=\frac{1}{2}$ and $h=1$, then we have
\begin{equation}\label{matrix}
   T_n=\left[\frac{2}{1+2(i-j)}\right]_{i, j=1}^n.
\end{equation}

In \cite{tyr}, Tyrtyshnikov obtained lower bounds for the spectral norm of the Cauchy--Toeplitz matrices as in (\ref{matrix}). The following upper and lower bounds for the $\ell_p$ norm of Cauchy--Toeplitz matrices as in (\ref{matrix}) was published in 1998 by Bozkurt \cite{bozkurt1}:

Let $p>1$, then 
\begin{equation}\label{bozkurteq1}
\left[\left(2^p-1\right) \zeta(p)\right]^{\frac{1}{p}} \leq n^{-\frac{1}{p}}\left\|T_n\right\|_p \leq 2^{\frac{1}{p}}\left[\left(2^p-1\right) \zeta(p)\right]^{\frac{1}{p}},
\end{equation}
where $\zeta(p)$ is Riemann zeta function.

Unfortunately, the first inequality is not true for $1<p<\varepsilon_n$  because of $\lim_{\varepsilon\to0}\zeta(1+\varepsilon)=\infty$. In the same paper, he conjectured the following. 
\begin{conjecture}
Let the matrix $T_n$ be as in (\ref{matrix}). Then the inequalities 
\begin{equation}\label{conj1}
n^{-\frac{1}{q}}\|T_n\|_{p,q}<4\left(\frac{1}{2}+\frac{1}{2^p-1}\right)^{\frac{1}{p}},~~(p\geq q)
\end{equation}
and 
\begin{equation}\label{conj2}
n^{-\frac{1}{q}}\|T_n\|_{p,q}\geq4\left(\frac{1}{2^p-1}\right)^{\frac{1}{p}},~~(p<q)
\end{equation}
are valid for the $\ell_{p,q}~(1\leq p,q\leq\infty)$ norm of the matrix $T_n$.
\end{conjecture}
For some related results of the norm of Cauchy--Toeplitz matrices, the reader is referred to papers \cite{gungor, bozkurt2, solak, bozkurt3,wu}.

The main objective of this paper is to derive new upper and lower bounds for the $\ell_p (1<p<\infty)$ norms of  Cauchy--Toeplitz matrices as in (\ref{matrix}). Moreover, we give a complete answer to the conjecture.
\section{Preliminaries}
In order to prove our main result, we need the following lemmas.

Let $A$ be any $n$-square matrix. The $\ell_p$ and $\ell_{p, q}(1 \leq p, q \leq \infty)$ norms of the matrix $A$ are
$$
\|A\|_p=\left[\sum_{i, j=1}^n\left|a_{i j}\right|^p\right]^{1 / p}
$$
and
$$
\|A\|_{p, q}=\left[\sum_{j=1}^n\left(\sum_{i=1}^n\left|a_{i j}\right|^p\right)^{\frac{q}{p}}\right]^{\frac{1}{q}}.
$$

For $r\in\mathbb{R}$, the power mean of the order $r$ of positive numbers $x_{1}, x_{2}, \ldots, x_{n}$ is defined by 
$$
M_{r}(\underline{x})=\left(\frac{1}{n} \sum_{k=1}^{n} x_{k}^{r}\right)^{\frac{1}{r}},
$$
where $\underline{x}=\left(x_{1}, x_{2}, \ldots, x_{n}\right)$. The power mean has the following nice properties (see \cite{bullen}).
\begin{lemma} 1. (Monotonicity)
If $r<s$, then
\begin{equation}\label{meanineq}
M_{r}(\underline{x}) \leq M_{s}(\underline{x})., 
\end{equation}
which is the well--known power mean inequality.\\
2. (Continuity) The power mean is a continuous function for $r$, that is
$$
\lim _{\varepsilon \rightarrow 0} M_{r+\varepsilon}(\underline{x})=M_{r}(\underline{x}).
$$
\end{lemma}
\begin{lemma}\label{lem2}
    Let $p\geq1$. Then the following inequality
    $$3-2^p+\frac{1}{3^p}-\left(\frac{2}{3}\right)^p>0$$
    holds for $1\leq p<\delta$. The opposite inequality holds $p>\delta$. Here $\delta=1.40485...$ is the unique root of $2^\delta + 6^\delta = 3^{\delta + 1} + 1$.
\end{lemma}
\begin{proof}
    Considering the function $g$ on $[1,+\infty)$ defined by 
    $$g(p)=3-2^p+\frac{1}{3^p}-\left(\frac{2}{3}\right)^p$$
    and differentiation yields
    \begin{align*}
     g'(p)&=-2^p\ln2-\frac{\ln3}{3^p}-\left(\frac{2}{3}\right)^p\ln\frac{2}{3}\\
     &=-2^p\left(\ln2-\frac{\ln3}{3^p}\right)-\frac{\ln3}{3^p}-\left(\frac{2}{3}\right)^p\ln2<0,
     \end{align*}
     which implies $g$ is decreasing on $[1,+\infty[$. From $g(1)=\frac{2}{3}, g(2)=-\frac{4}{3}$ and $g$ is continuous on $[1,+\infty[$, there is a unique $\delta\in(1,2)$ to satisfy equation $g(\delta)=0$. Therefore, $g(p)>0$ for $1\leq p<\delta$ and $g(p)<0$ for $p>\delta$, which proves the desired result.
\end{proof}
\begin{lemma}\label{lem3}
    Let $p>1$. Then the following inequality
    $$\left(1-\frac{1}{2^p}\right)\zeta(p)>2^{p-1}\left(\frac{1}{2}+\frac{1}{2^p-1}\right)$$
    holds for $1< p<\mu$. The opposite inequality holds for $p>\mu$. Here $\mu=1.6181...$ is the unique root of equation $$\left(1-\frac{1}{2^\mu}\right)\zeta(\mu)=2^{\mu-1}\left(\frac{1}{2}+\frac{1}{2^\mu-1}\right).$$
\end{lemma}
\begin{proof} Let 
$$f(p)=\left(1-\frac{1}{2^p}\right)\zeta(p)-2^{p-1}\left(\frac{1}{2}+\frac{1}{2^p-1}\right).$$
We show that the function $f(p)$ is strictly decreasing on $(1,+\infty)$.  Differentiation yields
$$f'(p)=-\sum_{n=2}^\infty\frac{\ln(2n-1)}{(2n-1)^p}-2^{p-2}\ln2\cdot\frac{2^{2p}-2^{p+1}-1}{(2^p-1)^2}.$$
To prove $f'(p)<0$ for $p>1$, we consider two cases.

{\bf Case $\boldsymbol{(1< p\leq\mu)}$}: Since a function $\frac{\ln(2x-1)}{(2x-1)^p}$ is strictly decreasing on $[2,\infty)$, we have
\begin{align*}
    f'(p)&<-\int_{2}^\infty\frac{\ln(2x-1)}{(2x-1)^p}\,dx-2^{p-2}\ln2\cdot\frac{2^{2p}-2^{p+1}-1}{(2^p-1)^2}\\
    &=-\frac{\ln3}{2(p-1)\cdot3^{p-1}}-\frac{1}{2(p-1)^2\cdot3^{p-1}}-2^{p-2}\ln2+\frac{2^{p-1}\ln2}{(2^p-1)^2}.
\end{align*}
Using an obvious inequalities $1+\ln3>2\ln2$, $2^{p-1}<2^p-1$ and $\frac{1}{2^p-1}-2^{p-2}<\frac{1}{2}$ lead to 
\begin{align*}
    f'(p)&<-\frac{1+\ln3}{2\cdot3^{p-1}}-2^{p-2}\ln2+\frac{2^{p-1}\ln2}{(2^p-1)^2}\\
    &<-\frac{\ln2}{3^{p-1}}-2^{p-2}\ln2+\frac{2^{p-1}\ln2}{(2^p-1)^2}\\
    &<-\frac{\ln2}{3^{p-1}}-2^{p-2}\ln2+\frac{\ln2}{2^p-1}\\
    &=\ln2\cdot \left(\frac{1}{2^p-1}-2^{p-2}-\frac{1}{3^{p-1}}\right)\\
    &<\ln2\cdot \left(\frac{1}{2}-\frac{1}{3^{p-1}}\right)\leq \ln2\cdot \left(\frac{1}{2}-\frac{1}{3^{\mu-1}}\right)\approx-0.005<0.
\end{align*}
{\bf Case $\boldsymbol{(p>\mu)}$}: Since $f_1(p)=2^{2p}-2^{p+1}-1$ is increasing on $(\mu,\infty)$, we obtain
$$2^{2p}-2^{p+1}-1>0.$$
This implies that $f'(p)<0$ for $p\in(\mu,\infty)$. 

Clearly, the function $f(p)$ is continuous in $(1,+\infty)$. So, there is a unique $\mu\in(1,\infty)$ to satisfy equation $f(\mu)=0$. Hence $f(p)>0$ for $1< p<\mu$ and $f(p)<0$ for $p>\mu$. The proof of Lemma \ref{lem3} is complete.
\end{proof}
\begin{lemma}\label{lem5}
    Let $p\geq1$. Then
   \begin{equation}\label{lem5eq1}
        n^{-1}\|T_n\|_{p,1}\leq (n+1)^{-1}\|T_{n+1}\|_{p,1}
   \end{equation}
    and 
       \begin{equation}\label{lem5eq2}
    \lim_{n\to \infty}n^{-1}\|T_n\|_{p,1}=2^\frac{1}{p}\left(\left(2^p-1\right)\zeta(p)\right)^\frac{1}{p},\quad p>1.
    \end{equation}
\end{lemma}
\begin{proof}
    Let $n_0=\left[\frac{n}{2}\right]+1$. Since $\sum_{i=1}^{n}|a_{im}|^p=\sum_{i=2}^{n+1}|a_{i(m+1)}|^p$, we have
    $$\sum_{i=1}^{n+1}|a_{i(m+1)}|^p>\sum_{i=1}^{n}|a_{im}|^p$$
    and 
       $$\sum_{i=1}^{n+1}|a_{im}|^p>\sum_{i=1}^{n}|a_{im}|^p.$$
       Moreover, it is easy to see that
       $$\max_{1\leq m\leq n}\sum_{i=1}^{n}|a_{im}|^p=\sum_{i=1}^{n}|a_{in_0}|^p.$$
    Thus, we have
    \begin{align*}
        n\|T_{n+1}\|_{p,1}-&(n+1)\|T_{n}\|_{p,1}\\&=n\sum_{j=1}^{n+1}\left(\sum_{i=1}^{n+1}|a_{ij}|^p\right)^\frac{1}{p}-(n+1)\sum_{j=1}^{n}\left(\sum_{i=1}^{n}|a_{ij}|^p\right)^\frac{1}{p}\\
        &=n\left(\sum_{\substack{1\leq j\leq n+1\\j\neq n_0}}\left(\sum_{i=1}^{n+1}|a_{ij}|^p\right)^\frac{1}{p}-\sum_{j=1}^{n}\left(\sum_{i=1}^{n}|a_{ij}|^p\right)^\frac{1}{p}\right)\\
        &\quad+n\left(\sum_{i=1}^{n+1}|a_{in_0}|^p\right)^\frac{1}{p}-\sum_{j=1}^{n}\left(\sum_{i=1}^{n}|a_{ij}|^p\right)^\frac{1}{p}\\
        &>n\left(\sum_{\substack{1\leq j\leq n+1\\j\neq n_0}}\left(\sum_{i=1}^{n+1}|a_{ij}|^p\right)^\frac{1}{p}-\sum_{j=1}^{n}\left(\sum_{i=1}^{n}|a_{ij}|^p\right)^\frac{1}{p}\right)\\
        &\quad+n\left(\sum_{i=1}^{n+1}|a_{in_0}|^p\right)^\frac{1}{p}-n\cdot\max_{1\leq j\leq n}\left(\sum_{i=1}^{n}|a_{ij}|^p\right)^\frac{1}{p}\\
        &=n\left(\sum_{\substack{1\leq j\leq n+1\\j\neq n_0}}\left(\sum_{i=1}^{n+1}|a_{ij}|^p\right)^\frac{1}{p}-\sum_{j=1}^{n}\left(\sum_{i=1}^{n}|a_{ij}|^p\right)^\frac{1}{p}\right)\\
        &\quad+n\left(\sum_{i=1}^{n+1}|a_{in_0}|^p\right)^\frac{1}{p}-n\left(\sum_{i=1}^{n}|a_{in_0}|^p\right)^\frac{1}{p}>0.
    \end{align*}
By the Stolz--Ces\`aro theorem, we get
    \begin{align*}
        \lim_{n\to \infty}n^{-1}\|T_n\|_{p,1}&=\lim_{n\to \infty}\frac{1}{n}\sum_{j=1}^{n}\left(\sum_{i=1}^{n}|a_{ij}|^p\right)^\frac{1}{p}\\
        &=\lim_{n\to \infty}\left[\sum_{j=1}^{n}\left(\sum_{i=1}^{n}|a_{ij}|^p\right)^\frac{1}{p}-\sum_{j=1}^{n-1}\left(\sum_{i=1}^{n-1}|a_{ij}|^p\right)^\frac{1}{p}\right]\\
        &=\lim_{n\to \infty}\left[\left(\sum_{i=1}^{n}|a_{in_0}|^p\right)^\frac{1}{p}+\sum_{j=1}^{n-1}\left(b_j^\frac{1}{p}-\left(\sum_{i=1}^{n-1}|a_{ij}|^p\right)^\frac{1}{p}\right)\right],
    \end{align*}
where $b_j=\begin{cases}
    \sum_{i=1}^{n}|a_{i(j+1)}|^p,& j\geq n_0\\
    \sum_{i=1}^{n}|a_{ij}|^p,& j<n_0
\end{cases}$. Using Lagrange's mean value theorem, we have
    $$
        b_j^\frac{1}{p}-\left(\sum_{i=1}^{n-1}|a_{ij}|^p\right)^\frac{1}{p}=\begin{cases}\frac{1}{p}|a_{1(j+1)}|^p\cdot C_{nj}^{\frac{1}{p}-1},& j\geq n_0\\
        \frac{1}{p}|a_{nj}|^p\cdot C_{nj}^{\frac{1}{p}-1},& j\geq n_0
        \end{cases} $$
    where, $\sum_{i=1}^{n-1}|a_{ij}|^p<C_{nj}<b_j$. Thus, we get that
       \begin{align*}
        0<\sum_{j=1}^{n-1}&\left(b_j^\frac{1}{p}-\left(\sum_{i=1}^{n-1}|a_{ij}|^p\right)^\frac{1}{p}\right)\\
        &=\frac{1}{p}\left(\sum_{j=1}^{n_0-1}|a_{nj}|^p\cdot C_{nj}^{\frac{1}{p}-1}+\sum_{j=n_0}^{n-1}|a_{1(j+1)}|^p\cdot C_{nj}^{\frac{1}{p}-1}\right)\\
        &<\frac{1}{p}\left(\sum_{j=1}^{n_0-1}|a_{nj}|^p\left(\sum_{i=1}^{n-1}|a_{ij}|^p\right)^{\frac{1}{p}-1}+\sum_{j=n_0}^{n-1}|a_{1(j+1)}|^p\left(\sum_{i=1}^{n-1}|a_{ij}|^p\right)^{\frac{1}{p}-1}\right)\\
        &<\frac{1}{p\cdot(n-1)^{1-\frac{1}{p}}}\left(\sum_{j=1}^{n_0-1}|a_{nj}|+\sum_{j=n_0}^{n-1}|a_{1(j+1)}|\right).
        \end{align*}
    Applying the Stolz--Ces\`aro theorem, we obtain
      \begin{align*}
    \lim_{n\to \infty}&\frac{1}{p\cdot(n-1)^{1-\frac{1}{p}}}\left(\sum_{j=1}^{n_0-1}|a_{nj}|+\sum_{j=n_0}^{n-1}|a_{1(j+1)}|\right)\\
    &=\frac{1}{p}\lim_{n\to \infty}\frac{1}{n^{1-\frac{1}{p}}-(n-1)^{1-\frac{1}{p}}}\left(\sum_{j=1}^{(n+1)_0-1}|a_{(n+1)j}|+\sum_{j=(n+1)_0}^{n}|a_{1(j+1)}|\right.\\
    &\qquad\qquad\left.-\sum_{j=1}^{n_0-1}|a_{nj}|-\sum_{j=n_0}^{n-1}|a_{1(j+1)}|\right)\\
    &=\frac{1}{p}\lim_{n\to \infty}\frac{1}{n^{1-\frac{1}{p}}-(n-1)^{1-\frac{1}{p}}}S_n=0,
       \end{align*}
since $S_n=\begin{cases}
    \frac{2}{2n-1},& \text{if} ~n~\text{is even}\\
    \frac{4n^2+1}{4n^3-n},& \text{if} ~n~\text{is odd}
\end{cases}$. Hence
       $$\lim_{n\to \infty}\sum_{j=1}^{n-1}\left(b_j^\frac{1}{p}-\left(\sum_{i=1}^{n-1}|a_{ij}|^p\right)^\frac{1}{p}\right)=0.$$
In view of above, we have that
     \begin{align*}
        \lim_{n\to \infty}n^{-1}\|T_n\|_{p,1}&=\lim_{n\to \infty}\left(\sum_{i=1}^{n}|a_{in_0}|^p\right)^\frac{1}{p}\\
        &=\lim_{n\to \infty}2\left(\sum_{i=1}^{n_0-1}\frac{2}{(2k-1)^p}+\frac{(-1)^{n+1}+1}{2(2n_0-1)^p}\right)^\frac{1}{p}\\
        &=2^\frac{1}{p}\left(\left(2^p-1\right)\zeta(p)\right)^\frac{1}{p}.
    \end{align*}
This completes the proof of Lemma \ref{lem5}.
\end{proof}
\section{Main Results}
From the definition of $\ell_p$ norm, we have
\begin{equation}\label{norm}
\left\|T_n\right\|_p=\left[\sum_{i, j=1}^n\left|a_{i j}\right|^p\right]^{\frac{1}{p}}=\left[2^p \sum_{k=1}^n \frac{2 n-2 k+1}{(2 k-1)^p}\right]^{\frac{1}{p}}.
\end{equation}
We begin with the following theorem that gives a lower bound of $\ell_p$ norm of $T_n$, which is independent of the size. 
\begin{theorem}\label{thm1}
Let $T_n$ be as in (\ref{matrix}), and let $p\geq1$. If $m\geq n$, then
\begin{equation}\label{generalineq1}
   2\leq n^{-\frac{1}{p}}\|T_n\|_{p}\leq m^{-\frac{1}{p}}\|T_m\|_{p}.
\end{equation}
\end{theorem}
\begin{proof} From the equality (\ref{norm}), we have
\begin{align*}
\frac{1}{n+1}&\|T_{n+1}\|_p^p-\frac{1}{n}\left\|T_{n}\right\|_p^p\\
&=\frac{2^p}{n+1}\sum_{k=1}^{n+1}\frac{2n-2k+3}{(2k-1)^p}-\frac{2^p}{n}\sum_{k=1}^{n}\frac{2n-2k+1}{(2k-1)^p}\\
&=2^p\left(\frac{1}{n+1}\sum_{k=1}^{n+1}\frac{2n-2k+3}{(2k-1)^p}-\frac{1}{n}\sum_{k=1}^{n}\frac{2n-2k+1}{(2k-1)^p}\right)\\
&=2^p\left(\frac{1}{(n+1)(2n+1)^p}+\sum_{k=1}^n\left(\frac{2n-2k+3}{(n+1)(2k-1)^p}-\frac{2n-2k+1}{n(2k-1)^p}\right)\right)\\
&=2^p\left(\frac{1}{(2n+1)^p}+\frac{1}{n(n+1)}\sum_{k=1}^n\frac{1}{(2k-1)^{p-1}}\right)>0.
\end{align*}
Thus, we have
$$2= \|T_1\|_{p}\leq n^{-\frac{1}{p}}\|T_n\|_{p}\leq m^{-\frac{1}{p}}\|T_m\|_{p}$$
\end{proof}
\begin{remark}\label{rem1}
    By the calculation in \cite{bozkurt1}, we have
$$\lim_{m\to\infty}m^{-\frac{1}{p}}\|T_m\|_{p}=2^{\frac{1}{p}}\left[(2^p-1)\zeta(p)\right]^{\frac{1}{p}},$$
hence 
$$  2\leq n^{-\frac{1}{p}}\|T_n\|_{p}<2^{\frac{1}{p}}\left[(2^p-1)\zeta(p)\right]^{\frac{1}{p}}.$$
We assume that $1<p<\mu$. Then, by the Lemma \ref{lem3}, there is a positive integer $N_{1,p}$ such that
$$N_{1,p}^{-\frac{1}{p}}\|T_{N_{1,p}}\|_{p}\leq 4\left(\frac{1}{2}+\frac{1}{2^p-1}\right)^{\frac{1}{p}}\leq (N_{1,p}+1)^{-\frac{1}{p}}\|T_{N_{1,p}+1}\|_{p},$$
for fixed $p$.
\end{remark}
\begin{remark}\label{rem2}
    Similarly to the above remark, if $1<p<\mu$, then by Lemma \ref{lem3} and Lemma \ref{lem5}, there is a positive integer $N_{2, p}$ such that 
    $$N_{2, p}^{-1}\|T_{N_{2, p}}\|_{p,1}\leq 4\left(\frac{1}{2}+\frac{1}{2^p-1}\right)^{\frac{1}{p}}\leq (N_{2, p}+1)^{-1}\|T_{N_{2, p}+1}\|_{p,1},$$
    for fixed $p$.  In Table 1, we give some values of $N_{1,p}$ and $N_{2,p}$.
\begin{table}[ht!]
\centering
\begin{tabular}{ |c|c|c|c|c|c|c|c|c|c|c|c| } 
\hline
$p$ & $1.1$ & $1.2$ & $1.3$ & $1.4$ & $1.5$ & $1.51$ & $1.52$ & $1.53$ & $1.54$ & $1.55$ & $1.56$\\
\hline
$N_{1,p}$ & $8$ & $10$ & $13$ & $20$ & $44$ & $49$ & $56$ & $65$ & $76$ & $92$ & $115$ \\
\hline
$N_{2,p}$ & $8$ & $10$ & $13$ & $20$ & $44$ & $50$ & $57$ & $66$ & $78$ & $94$ & $118$\\
\hline
\end{tabular}
\caption{}
\label{table:1}
\end{table}
\end{remark}
The next theorem gives new bounds of the $\ell_p$ norm of $T_n$ that depend on $n$ and $p$.
\begin{theorem}\label{thm2}
Let $T_n$ be as in (\ref{matrix}), and let $p>1, n\geq2$. Then
\begin{equation}\label{generalineq2}
   \left[\left(1+\frac{2}{n^2}\right)(2^p-1)\zeta(p)-C'(n,p)\right]^\frac{1}{p}< n^{-\frac{1}{p}}\|T_n\|_{p}<\left(2-\frac{1}{n}\right)^{\frac{1}{p}}\left[(2^p-1)\zeta(p)-C''(n,p)\right]^{\frac{1}{p}},
\end{equation} where
$$C'(n,p)=\frac{n^2+2}{(p-1)n^{p+1}}+\frac{2^{p+1}}{n(2n-1)^p},\quad C''(n,p)=\frac{2^{p-1}}{(p-1)(2n+1)^{p-1}}$$
\end{theorem}
\begin{proof} By the equality (\ref{norm}) and using the Binet--Cauchy identity with $a_k=\frac{1}{n}, b_k=2n-2k+1, c_k=\frac{1}{n}, d_k=\frac{1}{(2k-1)^p}$, we have
\begin{align}\label{thm2eq1}
\|T\|_p&=2\left[\sum_{k=1}^n\frac{2n-2k+1}{(2k-1)^p}\right]^\frac{1}{p}\notag\\
&=2\left[\frac{1}{n}\left(\sum_{k=1}^nb_k\right)\left(\sum_{k=1}^nd_k\right)+\frac{1}{n}\sum_{1\leq k<m\leq n}(b_m-b_k)(d_m-d_k)\right]^\frac{1}{p}\notag\\
&=2\left[n\sum_{k=1}^n\frac{1}{(2k-1)^p}+\frac{2}{n}\sum_{k=1}^{n-1}\sum_{m=k+1}^{n}(m-k)\left(\frac{1}{(2k-1)^p}-\frac{1}{(2m-1)^p}\right)\right]^\frac{1}{p}.
\end{align}
Since $d_k\geq d_{m-1}$ and $m-k\geq1$ for $1\leq k<m\leq n$, we have
\begin{align*}
    \sum_{k=1}^{n-1}\sum_{m=k+1}^{n}&(m-k)\left(\frac{1}{(2k-1)^p}-\frac{1}{(2m-1)^p}\right)\\
    &> \sum_{k=1}^{n-1}\sum_{m=k+1}^{n}\left(\frac{1}{(2m-3)^p}-\frac{1}{(2m-1)^p}\right)\\
    &= \sum_{k=1}^{n-1}\left(\frac{1}{(2k-1)^p}-\frac{1}{(2n-1)^p}\right)\\
    &= \sum_{k=1}^n\frac{1}{(2k-1)^p}-\frac{n}{(2n-1)^p},
\end{align*}
hence
$$\|T\|_p>2n^\frac{1}{p}\left[\left(1+\frac{2}{n^2}\right)\sum_{k=1}^n\frac{1}{(2k-1)^p}-\frac{2}{n(2n-1)^p}\right]^\frac{1}{p}.$$
On the other hand, the function $\frac{1}{(2x-1)^p}$ is a convex on interval $[k-\frac{1}{2}, k+\frac{1}{2} ]$, application of the well-known Hermite--Hadamard inequality on intervals $[k-\frac{1}{2}, k+\frac{1}{2} ],~k\geq n+1$ yields
\begin{align*}
    \sum_{k=1}^n\frac{1}{(2k-1)^p}&=\sum_{k=1}^\infty\frac{1}{(2k-1)^p}-\sum_{k=n+1}^\infty\frac{1}{(2k-1)^p}\\
    &\geq \left(1-\frac{1}{2^p}\right)\zeta(p)-\sum_{k=n+1}^\infty\int_{k-\frac{1}{2}}^{k+\frac{1}{2}}\frac{1}{(2x-1)^p}dx\\
    &= \left(1-\frac{1}{2^p}\right)\zeta(p)-\int_{n+\frac{1}{2}}^\infty\frac{1}{(2x-1)^p}dx\\
    &=\left(1-\frac{1}{2^p}\right)\zeta(p)-\frac{1}{2^{p}(p-1)n^{p-1}}.
\end{align*}
Then it follows
$$n^{-\frac{1}{p}}\|T\|_p>\left[\left(1+\frac{2}{n^2}\right)(2^p-1)\zeta(p)-\frac{n^2+2}{(p-1)n^{p+1}}-\frac{2^{p+1}}{n(2n-1)^p}\right]^\frac{1}{p}.$$
Now we shall prove the last inequality in (\ref{generalineq2}). From (\ref{thm2eq1}), we have
\begin{align*}
\left\|T_n\right\|_p&=2\left[n\sum_{k=1}^n\frac{1}{(2k-1)^p}+\frac{2}{n}\sum_{k=1}^{n-1}\sum_{m=k+1}^{n}(m-k)\left(\frac{1}{(2k-1)^p}-\frac{1}{(2m-1)^p}\right)\right]^\frac{1}{p}\\
&< 2\left[n\sum_{k=1}^n\frac{1}{(2k-1)^p}+\frac{2}{n}\sum_{k=1}^{n-1}\frac{1}{(2k-1)^p}\sum_{m=k+1}^{n}(m-k)\right]^\frac{1}{p}\\
&= 2\left[n\sum_{k=1}^n\frac{1}{(2k-1)^p}+\frac{1}{n}\sum_{k=1}^{n-1}\frac{(n-k)(n-k+1)}{(2k-1)^p}\right]^\frac{1}{p}\\
&\leq 2\left[n\sum_{k=1}^n\frac{1}{(2k-1)^p}+\frac{1}{n}\sum_{k=1}^{n}\frac{(n-1)n}{(2k-1)^p}\right]^\frac{1}{p}\\
&= 2n^\frac{1}{p}\left[\left(2-\frac{1}{n}\right)\sum_{k=1}^n\frac{1}{(2k-1)^p}\right]^\frac{1}{p}.
\end{align*}
By the decreasing property of the function $\frac{1}{(2x-1)^p}$, it follows that
\begin{align*}
    \sum_{k=1}^n\frac{1}{(2k-1)^p}&=\sum_{k=1}^\infty\frac{1}{(2k-1)^p}-\sum_{k=n+1}^\infty\frac{1}{(2k-1)^p}\\
    &<\left(1-\frac{1}{2^p}\right)\zeta(p)-\int_{n+1}^\infty\frac{1}{(2x-1)^p}dx\\
    &=\left(1-\frac{1}{2^p}\right)\zeta(p)-\frac{1}{2(p-1)(2n+1)^{p-1}}.
\end{align*}
Thus, we get
$$n^{-\frac{1}{p}}\left\|T_n\right\|_p<\left(2-\frac{1}{n}\right)^{\frac{1}{p}}\left[(2^p-1)\zeta(p)-\frac{2^{p-1}}{(p-1)(2n+1)^{p-1}}\right]^\frac{1}{p}.$$
The proof is complete.
\end{proof}
\begin{remark}
It should be noticed here that the right-hand side of (\ref{generalineq2}) improves the right-hand side of (\ref{bozkurteq1}).
\end{remark}
The following two theorems give an answer to the conjecture.
\begin{theorem}\label{mainthm1}
Let $T_n$ be as in (\ref{matrix}), and let $1\leq p<q\leq\infty$. Then the following inequality
\begin{equation}\label{mainthmeq1}
n^{-\frac{1}{q}}\left\|T_{n}\right\|_{p,q}\geq 4\left(\frac{1}{2^{p}-1}\right)^{\frac{1}{p}}, 
\end{equation}
holds for $n\geq3$ and for $n=2, (p\geq\delta~~\text{or}~~1<p<\delta, q\geq \delta_p)$. The opposite inequality holds for $n=1$ and for $n=2, (p=1~~\text{or}~~1< p<\delta, q<\delta_p )$. Here $q=\delta_p$ is the unique root of the following equation
\begin{equation}\label{delta}
2^{-p}\left(\left(1+\frac{1}{3^p}\right)^\frac{q}{p}+2^\frac{q}{p}\right)^p=2^{pq}\left(\frac{1}{2^p-1}\right)^q,
\end{equation}
for fixed $p$. 
\end{theorem}
\begin{proof} {\bf Case $\boldsymbol{(n\geq3)}$}: By the power mean inequality, we have
\begin{align}\label{monoeq1}
n^{-\frac{1}{q}}\left\|T_{n}\right\|_{p,q}&=n^{-\frac{1}{q}}\left[\sum_{j=1}^n\left(\sum_{i=1}^n|a_{ij}|^p\right)^\frac{q}{p}\right]^\frac{1}{q}=\left[\frac{1}{n}\sum_{j=1}^n\left(\sum_{i=1}^n|a_{ij}|^p\right)^\frac{q}{p}\right]^\frac{1}{q}\notag\\
&=M_q\left(\underline{\left(\sum_{i=1}^n|a_{ij}|^p\right)^\frac{1}{p}}\right)\notag\\
&>M_p\left(\underline{\left(\sum_{i=1}^n|a_{ij}|^p\right)^\frac{1}{p}}\right)=\left[\frac{1}{n}\sum_{j=1}^n\left(\sum_{i=1}^n|a_{ij}|^p\right)^\frac{p}{p}\right]^\frac{1}{p}\notag\\
&=n^{-\frac{1}{p}}\left[\sum_{j=1}^n\sum_{i=1}^n|a_{ij}|^p\right]^\frac{1}{p}=n^{-\frac{1}{p}}\|T_n\|_p.
\end{align}  
Applying  (\ref{generalineq1}), gives
\begin{equation}\label{mainthm1eq2}
n^{-\frac{1}{p}}\|T_n\|_p>3^{-\frac{1}{p}}\left\|T_{3}\right\|_{p}=3^{-\frac{1}{p}}\left(2^{p}\left(5+\frac{1}{3^{p-1}}+\frac{1}{5^{p}}\right)\right)^{\frac{1}{p}}.
\end{equation}
Combining (\ref{monoeq1}) and (\ref{mainthm1eq2}) leads to
$$
n^{-\frac{1}{q}}\left\|T_{n}\right\|_{p,q} >3^{-\frac{1}{p}}\left(2^{p}\left(5+\frac{1}{3^{p-1}}+\frac{1}{5^{p}}\right)\right)^{\frac{1}{p}}.
$$
To prove the inequality (\ref{mainthmeq1}) for $n\geq3$, it suffices to show that the following inequality
$$
3^{-\frac{1}{p}}\left(2^{p}\left(5+\frac{1}{3^{p-1}}+\frac{1}{5^{p}}\right)\right)^{\frac{1}{p}}>4\left(\frac{1}{2^{p}-1}\right)^{\frac{1}{p}},
$$
i.e.,
$$
2^{p+1}-5+3\left(\frac{2}{3}\right)^{p}+\left(\frac{2}{5}\right)^{p}-\frac{1}{3^{p-1}}-\frac{1}{5^p}>0.
$$
Considering the function $h$ on $[1,+\infty)$ defined by 
$$
h(p)=2^{p+1}-5+3\left(\frac{2}{3}\right)^{p}+\left(\frac{2}{5}\right)^{p}-\frac{1}{3^{p-1}}-\frac{1}{5^p}.
$$
Differentiation yields
\begin{align*}
h'(p)&=2^{p+1} \ln 2+3\left(\frac{2}{3}\right)^{p} \ln \frac{2}{3}+\left(\frac{2}{5}\right)^{p} \ln \frac{2}{5}+\frac{1}{3^{p-1}}\ln3+\frac{1}{5^p}\ln 5\\
&>2^{p+1} \ln 2+3\left(\frac{2}{3}\right)^{p} \ln \frac{2}{3}+\left(\frac{2}{5}\right)^{p} \ln \frac{2}{5}:=h_1(p)
\end{align*}
and 
$$
h_1'(p)=  2^{p+1} (\ln 2)^2+3\left(\frac{2}{3}\right)^{p} \left(\ln \frac{2}{3}\right)^2+\left(\frac{2}{5}\right)^{p} \left(\ln \frac{2}{5}\right)^2>0.
$$
Thus, we deduce that for $p\geq1$,
 $$
h_1(p) \geq h_1(1)=4\ln 2+2\ln \frac{2}{3}+\frac{2}{5}\ln \frac{2}{5}\approx1.5951\ldots>0.
$$
Therefore, the function $h(p)$ is increasing on $[1,+\infty[$,  which implies that 
$$h(p)>h(1)=\frac{1}{5}>0$$
for $p\geq1$.

{\bf Case $\boldsymbol{(n=2, p\geq\delta)}$:}  From the inequality (\ref{monoeq1}), we obtain
$$
2^{-\frac{1}{q}}\left\|T_{2}\right\|_{p,q}>2^{-\frac{1}{p}}\left\|T_{2}\right\|_{p}.
$$
Clearly, it suffices to show that 
\begin{equation}\label{eq12}
2^{-\frac{1}{p}}\left\|T_{2}\right\|_{p} \geq 4\left(\frac{1}{2^{p}-1}\right)^{\frac{1}{p}},
\end{equation}
This inequality is equivalent to
 $$3-2^p+\frac{1}{3^p}-\left(\frac{2}{3}\right)^p\leq0,$$
which is true by the Lemma \ref{lem2}.

 {\bf Case $\boldsymbol{(n=2, 1\leq p<\delta)}$:} If $p=1$, then by the power mean inequality, we have
$$2^{-\frac{1}{q}}\left\|T_{2}\right\|_{1, q}=2^{1-\frac{1}{q}}\left(\left(\frac{4}{3}\right)^q+2^q\right)^\frac{1}{q}<\lim_{q\to\infty}2^{1-\frac{1}{q}}\left(\left(\frac{4}{3}\right)^q+2^q\right)^\frac{1}{q}=4.$$
In other words, the inequality (\ref{mainthmeq1}) does not hold for $p=1$.
 
Similarly to the previous case, one has for $1\leq p<\delta$,
\begin{equation}\label{thmeqq2}
2^{-\frac{1}{p}}\left\|T_{2}\right\|_{p} < 4\left(\frac{1}{2^{p}-1}\right)^{\frac{1}{p}}.
\end{equation}

If $1< p<\delta$, then by the power mean inequality, we have  
\begin{align*}
2^{-\frac{1}{q}}\left\|T_{2}\right\|_{p,q}&=2^{1-\frac{1}{q}}\left(\left(1+\frac{1}{3^p}\right)^\frac{q}{p}+2^\frac{q}{p}\right)^\frac{1}{q}\\
&<\lim_{q\to+\infty}2^{1-\frac{1}{q}}\left(\left(1+\frac{1}{3^p}\right)^\frac{q}{p}+2^\frac{q}{p}\right)^\frac{1}{q}=2^{1+\frac{1}{p}}.
\end{align*}
It is easy to see that
$$4\left(\frac{1}{2^{p}-1}\right)^{\frac{1}{p}}<2^{1+\frac{1}{p}}.$$
By the continuity and monotonicity of power mean, there are $q'$ and $q''$ such that $p<q'\leq q''$ and 
$$
2^{-\frac{1}{p}}\left\|T_{2}\right\|_{p}<2^{-\frac{1}{q}}\left\|T_{2}\right\|_{p, q'}\leq4\left(\frac{1}{2^p-1}\right)^{\frac{1}{p}}\leq2^{-\frac{1}{q}}\left\|T_{2}\right\|_{p, q''}<2^{1+\frac{1}{p}}.
$$
Hence, there is exist a unique $q=\delta_p$ for fixed $p$, such that
$$2^{-\frac{1}{q}}\left\|T_{2}\right\|_{p,q}=4\left(\frac{1}{2^p-1}\right)^{\frac{1}{p}},$$
that is,
$$2^{-p}\left(\left(1+\frac{1}{3^p}\right)^\frac{\delta_p}{p}+2^\frac{\delta_p}{p}\right)^p=2^{p\delta_p}\left(\frac{1}{2^p-1}\right)^{\delta_p}.$$
Therefore, by the power mean inequality, we have
$$
2^{-\frac{1}{q}}\left\|T_{2}\right\|_{p,q}<4\left(\frac{1}{2^{p}-1}\right)^{\frac{1}{p}},
$$
for $1<q<\delta_p$, and 
$$
2^{-\frac{1}{q}}\left\|T_{2}\right\|_{p,q}\geq4\left(\frac{1}{2^{p}-1}\right)^{\frac{1}{p}},
$$
for $q\geq\delta_p$. We thus complete the proof.
\end{proof}
\begin{remark} In Table 2, we give some values of $\delta_p$.
\begin{table}[ht!]
\centering
\begin{tabular}{ |c|c|c|c|c|c|c|c| } 
\hline
$p$ & $1.1$ & $1.15$ & $1.2$ & $1.25$ & $1.3$ & $1.35$ & $1.4$\\
\hline
$\delta_p$ & $11.5839...$ & $7.9001...$ & $5.8018...$ & $4.3471...$ & $3.2220...$ & $2.2895...$ & $1.4787...$ \\
\hline
\end{tabular}
\caption{}
\label{table:2}
\end{table}
\end{remark}
\begin{theorem}\label{mainthm2}
Let $T_n$ be as in (\ref{matrix}), and let $1\leq q \leq p\leq\infty$. Then the following inequality
\begin{equation}\label{mainthm2eq0}
 n^{-\frac{1}{q}}\|T_n\|_{p,q}<4\left(\frac{1}{2}+\frac{1}{2^p-1}\right)^{\frac{1}{p}},    
\end{equation}
holds for 
\begin{itemize}
    \item $p=1, 1\leq n\leq7$
    \item $\mu\leq p$
    \item $1<p<\mu, n\leq N_{1,p}, (N_{1,p}\geq 7)$
    \item $1<p<\mu, N_{1,p}<n\leq N_{2, p}, 1\leq q<\eta_{p,n}$.
\end{itemize} The opposite inequality holds for other values of $n, p, q$. Here $q=\eta_{p,n}$ is the unique root of the following
$$n^{-\frac{1}{q}}\|T_{n}\|_{p,q}= 4\left(\frac{1}{2}+\frac{1}{2^p-1}\right)^{\frac{1}{p}},$$
for fixed $p~(1<p<\mu)$ and $n~(N_{1,p}<n\leq N_{2,p})$.
\end{theorem}
\begin{proof} By the power mean inequality, we have
\begin{equation}\label{mainthm2eq1}
n^{-1}\|T_n\|_{p,1}\leq n^{-\frac{1}{q}}\|T_n\|_{p,q}\leq n^{-\frac{1}{p}}\|T_n\|_{p}.
\end{equation}
To prove our result, we consider four cases.

{\bf Case $\boldsymbol{(p=1)}$}: We have 
$$  \lim_{n\to\infty}n^{-1}\|T_n\|_{1,1}=\lim_{n\to\infty}n^{-1}\|T_n\|_{1}=\lim_{n\to\infty}\left[2\left(\sum_{k=1}^n\frac{1}{2k-1}\right)-1\right]=\infty.
$$
Hence, by the inequality (\ref{lem5eq1}) and simple computation, we get
$$n^{-1}\|T_n\|_{1,1}<6$$
for $n\leq7$, and
$$n^{-1}\|T_n\|_{1,1}>6$$
for $n\geq8$.

{\bf Case $\boldsymbol{(\mu\leq p)}$}: By inequality (\ref{mainthm2eq1}) and Theorem \ref{thm1}, we have
$$n^{-\frac{1}{q}}\|T_n\|_{p,q}\leq n^{-\frac{1}{p}}\|T_n\|_{p}<2^{\frac{1}{p}}\left[(2^p-1)\zeta(p)\right]^{\frac{1}{p}}.$$
If we show that 
$$2^{\frac{1}{p}}\left[(2^p-1)\zeta(p)\right]^{\frac{1}{p}}\leq4\left(\frac{1}{2}+\frac{1}{2^p-1}\right)^{\frac{1}{p}},$$
then the desired result follows. This inequality is equivalent to
$$\left(1-\frac{1}{2^p}\right)\zeta(p)\leq 2^{p-1}\left(\frac{1}{2}+\frac{1}{2^p-1}\right),$$
which is true by the Lemma \ref{lem3}.

{\bf Case $\boldsymbol{(1< p< \mu, n\leq N_{1,p})}$}: By inequality (\ref{mainthm2eq1}), Theorem \ref{thm1} and Remark \ref{rem1}, we have
 $$n^{-\frac{1}{q}}\|T_n\|_{p,q}\leq n^{-\frac{1}{p}}\|T_n\|_{p}\leq N_{1,p}^{-\frac{1}{p}}\|T_{N_{1,p}}\|_{p}\leq4\left(\frac{1}{2}+\frac{1}{2^p-1}\right)^{\frac{1}{p}}.$$ 
Now, we will prove that $N_{1,p}\geq7$. By inequality (\ref{mainthm2eq1}) and Theorem \ref{thm1}, we have
$$n^{-\frac{1}{q}}\|T_n\|_{p,q}\leq n^{-\frac{1}{p}}\|T_n\|_{p}\leq7^{-\frac{1}{p}}\|T_7\|_{p}.$$
Hence it is enough to prove that
$$7^{-\frac{1}{p}}\|T_7\|_{p}<4\left(\frac{1}{2}+\frac{1}{2^p-1}\right)^{\frac{1}{p}}.$$
This inequality equivalent to
$$7\cdot 2^{p-1}+\frac{7\cdot 2^p}{2^p-1}-13-\frac{11}{3^p}-\frac{9}{5^p}-\frac{1}{7^{p-1}}-\frac{5}{9^p}-\frac{3}{11^p}-\frac{1}{13^p}>0.$$
Let 
$$\psi(p)=7\cdot 2^{p-1}+\frac{7\cdot 2^p}{2^p-1}-13-\frac{11}{3^p}-\frac{9}{5^p}-\frac{1}{7^{p-1}}-\frac{5}{9^p}-\frac{3}{11^p}-\frac{1}{13^p}.$$
Differentiation yields
$$\psi'(p)=7\cdot 2^{p-1}\ln 2-\frac{7\cdot 2^p\ln2}{(2^p-1)^2}+\frac{11\ln3}{3^p}+\frac{9\ln5}{5^p}+\frac{\ln7}{7^{p-1}}+\frac{5\ln9}{9^p}+\frac{3\ln11}{11^p}+\frac{\ln13}{13^p}.$$
Using an obvious inequalities $\ln3>\ln2$, $\ln 5>2\ln2$, $\ln 7>2\ln2$ and $2^{p-1}<2^p-1$ lead to
\begin{align*}
\psi'(p)&>7\cdot 2^{p-1}\ln 2-\frac{7\cdot 2^p\ln2}{(2^p-1)^2}+\frac{11\ln3}{3^p}+\frac{9\ln5}{5^p}+\frac{\ln7}{7^{p-1}}\\
&>7\cdot 2^{p-1}\ln 2-\frac{7\cdot 2^p\ln2}{(2^p-1)^2}+\frac{11\ln2}{3^p}+\frac{18\ln2}{5^p}+\frac{2\ln2}{7^{p-1}}\\
&>\left(7\cdot 2^{p-1}-\frac{7}{2^{p-2}}+\frac{11}{3^p}+\frac{18}{5^p}+\frac{2}{7^{p-1}}\right)\ln2\\
&>\left(7\cdot 2^{p-1}-14+\frac{11}{3^p}+\frac{18}{5^p}+\frac{2}{7^{p-1}}\right)\ln2:=\psi_1(p)\ln2.
\end{align*}
Next, we show that $\psi_1(p)>0$ for $1\leq p\leq \mu$. Differentiation yields
$$\psi'_1(p)=7\cdot 2^{p-1}\ln2-\frac{11\ln3}{3^p}-\frac{18\ln5}{5^p}-\frac{2\ln7}{7^{p-1}},$$
$$\psi''_1(p)=7\cdot 2^{p-1}(\ln2)^2+\frac{11(\ln3)^2}{3^p}+\frac{18(\ln5)^2}{5^p}+\frac{2(\ln7)^2}{7^{p-1}}>0.$$
This reveals that for $1\leq p\leq \mu$,
$$\psi_1(p)\geq \psi_1(1.5)+\psi'_1(1.5)(p-1.5)\approx 0.38+0.47(p-1.5)>0.$$
Thus, $\psi(p)$ is strictly increasing on $[1,\mu]$. Therefore, we get
$$\psi(p)>\psi(1)=\frac{4042}{6435}>0,$$
 for $1< p< \mu$.

{\bf Case $\boldsymbol{(1< p\leq \nu, n>N_{1,p})}$}: By inequality (\ref{mainthm2eq1}), we get $N_{2, p}\geq N_{1,p}$. We consider two sub-cases.

Case ($N_{1,p}<n\leq N_{2, p}$): By the Remark \ref{rem1} and Remark \ref{rem2}, we have
$$n^{-1}\|T_{n}\|_{p,1}\leq 4\left(\frac{1}{2}+\frac{1}{2^p-1}\right)^{\frac{1}{p}}\leq n^{-\frac{1}{p}}\|T_{n}\|_{p}.$$
From the inequality (\ref{mainthm2eq1}) and continuity of the power means, there is unique number $\eta_{p,n}$ such that $1<\eta_{p,n}<p$ and 
$$n^{-\frac{1}{\eta_{p,n}}}\|T_{n}\|_{p,\eta_{p,n}}= 4\left(\frac{1}{2}+\frac{1}{2^p-1}\right)^{\frac{1}{p}},$$
for fixed $p$. Thus, by monotonicity of power means
$$n^{-\frac{1}{q}}\|T_{n}\|_{p,q}<4\left(\frac{1}{2}+\frac{1}{2^p-1}\right)^{\frac{1}{p}},$$
for $q<\eta_{p,n}$ and
$$n^{-\frac{1}{q}}\|T_{n}\|_{p,q}\geq4\left(\frac{1}{2}+\frac{1}{2^p-1}\right)^{\frac{1}{p}},$$
for $\eta_{p,n}\leq q\leq p$.

Case ($N_{2, p}<n$): By the inequality (\ref{mainthm2eq1}), Lemma \ref{lem5} and Remark \ref{rem2}, we have
$$n^{-\frac{1}{q}}\|T_{n}\|_{p,q}\geq n^{-1}\|T_n\|_{p,1}\geq (N_{2,p}+1)^{-1}\|T_{N_{2,p}+1}\|_{p,1}\geq4\left(\frac{1}{2}+\frac{1}{2^p-1}\right)^{\frac{1}{p}}.$$
The proof is complete.
\end{proof}
\begin{remark}
    In Table 3, we give some values of $\eta_{p, n}$.
\begin{table}[ht!]
\centering
\begin{tabular}{ |c|c|c|c|c|c| } 
\hline
$p$ & $1.51$ & $1.52$ & $1.53$ & $1.54$ & $1.54$ \\
\hline
$n$ & $50$ & $57$ & $66$ & $77$ & $78$ \\
\hline
$\eta_{p, n}$& $1.2369...$ & $1.1900...$ & $1.1448...$ & $1.4360...$ & $1.0739...$ \\
\hline\hline
$p$ &  $1.55$ & $1.55$ & $1.56$ & $1.56$ & $1.56$\\
\hline
$n$  & $93$ & $94$ & $116$ & $117$ & $118$ \\
\hline
$\eta_{p, n}$ & $1.4701..$ & $1.1557...$ & $1.5451...$ & $1.2763...$ & $1.0216...$ \\
\hline
\end{tabular}
\caption{}
\label{table:3}
\end{table}
\end{remark}

\end{document}